\documentclass{birkjour}
\usepackage{float}
\usepackage{upgreek}
\usepackage[table]{xcolor}

\usepackage{amssymb}
\usepackage{graphicx}
\usepackage{epstopdf}
\usepackage{latexsym}
\usepackage{amsmath}
\usepackage{amsfonts}
\usepackage{array}
\usepackage{microtype}
\usepackage{amsthm}
\newtheorem{definition}{\bf Definition}[section]
\newtheorem{lemma}[definition]{\bf Lemma}
\newtheorem{theorem}[definition]{\bf Theorem}
\newtheorem{proposition}[definition]{\bf Proposition}
\newtheorem{corollary}[definition]{\bf Corollary}

\newtheorem{remark}[definition]{\bf Remark}
\newtheorem{example}[definition]{\bf Example}
\usepackage{xcolor}

\begin{document}


\title[Minimal commutants, bicommutants for analytic Toeplitz operators]{Minimal commutant and double commutant property for analytic Toeplitz operators}

\author{María José  Gonz\'alez}
\address{Departamento de Matemáticas, Universidad de Cádiz, 11510 Puerto Real (Cádiz), Spain} 
\email{majose.gonzalez@uca.es}

\author{Fernando Le\'on-Saavedra}
\address{ Departamento de Matemáticas and INDESS, Universidad de Cádiz, 11405 Jerez de la Frontera (Cádiz), Spain} \email{fernando.leon@uca.es}

\subjclass{Primary 47B35; Secondary 47B38, 30J05}

\keywords{Analytic Toeplitz operator, Minimal commutant, Double commutant property.}

\date{January 1, 2004}

\begin{abstract}
In this paper, we study the minimality of the commutant of an analytic Toeplitz operator
$M_\varphi$, when $M_\varphi$ is defined on the Hardy space
$H^2(\mathbb{D})$ and $\varphi\in H^\infty(\mathbb{D})$ denotes a bounded analytic function in $\mathbb{D}$. Specifically, we show that the commutant of $M_\varphi$ is 
minimal if and only if the polynomials on $\varphi$ are weak star dense in 
$H^\infty(\mathbb{D})$, that is, $\varphi$ is a weak star generator of 
$H^\infty(\mathbb{D})$. We use our result to characterize when the double commutant
of an analytic Toeplitz operator $M_\varphi$ is minimal for a large class of
symbols $\varphi$. Specifically,
when $\varphi$ is an entire function, or more generally, when $\varphi$ belongs to the Thomson-Cowen class.
\end{abstract}

\maketitle



\section{Introduction}

Let $\mathcal{L}(H)$ be the algebra of all bounded linear operators on a
Hilbert space $H$. If $T\in \mathcal{L}(H)$, then the following algebras, where containing
$T$, arise naturally on $\mathcal{L}(H)$: $\overline{\textrm{Alg}(T)}^{WOT}$, $\{T\}'$, and $\{T\}''$. When $\textrm{Alg}(T)$ denotes the unital algebra
generated by $T$, and $\overline{\textrm{Alg}(T)}^{WOT}$  its closure in the weak operator topology. The commutant of $T$, denoted by $\{T\}'$, is the set of all operators $A$, such that $AT=TA$.  The double commutant of $T$, denoted by $\{T\}''$, is the set of all operators $A$ such that $AB=BA$ for all $B\in \{T\}'$.

In general, the following inclusions $\textrm{Alg}(T)\subset 
\{T\}''\subset \{T\}'$ are trivial, and since the algebras $\{T\}', \{T\}''$ are
closed in the weak operator topology, we also have $\overline{\textrm{Alg}(T)}^{WOT}
\subset \{T\}''\subset \{T\}'$. 
The commutant and double commutant of an operator determine its structure. Therefore, it is important to know the commutant of a given operator. For example, it is natural to ask when $\overline{\textrm{Alg}
(T)}^{WOT}=\{T\}'$, that is, when $T$ has {\it the minimal commutant property},  because when $T$ enjoys this property, then the invariant and hyperinvariant subspace lattices of $T$ coincide. In addition, $T$ is said to have the {\it double  commutant property} if its double commutant is equal to the weak closure of the polynomials on $T$.

The literature on these topics is huge. In general operator theory, we cannot talk about double commutant without citing J. Von Neumann's double commutant Theorem \cite{neumann}. 
The characterization of the norm closure of operators with the minimal commutant property remains an open problem (see \cite{herrero}). This question, which still resists, first appeared at the end of Turner's Ph.D. thesis (\cite{turner}), but also, as Herrero pointed out, it was formulated in private communication by J.B. Conway.

Within the study of special classes of linear operators, 
resolving questions regarding the commutants of analytic Toeplitz operators defined on $H^2(\mathbb{D})$ was the main focus of research in the early 1970s.

The main ancestors of this paper are the results of Abrahamse, Baker, Cowen, Deddens, Shields, Thomson, and many others (\cite{cowen,deddenswong,shieldswallen,thomson,turner}). The study of commutants and double commutants of composition operators is more recent. The first findings appeared in the 1990s \cite{carter, cload,worner}. More recently, composition operators, induced by linear fractional self-maps of the unit disk,  having the minimal commutant property \cite{mcp} and the double commutant property \cite{dcp} have been characterized.

Turner, in his Ph.D. thesis \cite{turner}, formulated the following question: 
\begin{quote}
{\bf Question 1:} For which maps $\varphi\in H^\infty(\mathbb{D})$ does the analytic Toeplitz operator $M_\varphi$ have the double commutant property?
\end{quote}
As we shall see in the present paper, Question 1 is related to the minimal commutant property for analytic Toeplitz operators. Thus, the following question arises:

\begin{quote}
{\bf Question 2:} For which maps $\varphi\in H^\infty(\mathbb{D})$ does the analytic Toeplitz operator $M_\varphi$ have the minimal commutant property?
\end{quote}

In this paper, we show that an analytic Toeplitz operator $M_\varphi$ has a minimal commutant if and only if its inducing map $\varphi$ is a weak star generator of $H^\infty(\mathbb{D})$. Recall that a map $\varphi\in H^\infty(\mathbb{D})$ is a weak star generator of $H^\infty(\mathbb{D})$ if and only if the polynomials on $\varphi$ are weak star dense in $H^\infty (\mathbb{D})$ (\cite{sarasondos,sarason}).
The reason for this phenomenon is that the closure of linear subspaces in $H^\infty(\mathbb{D})$ in some different topologies coincide. The proof uses Banach and Mazurkiewicz's sequential approach of weak star topology (\cite[Annexe]{banach}).

As a consequence of this {\it Toeplitz minimal commutant Theorem}, we could characterize when a univalent symbol $\varphi$ induces a Toeplitz operator with the double commutant property.
Moreover, with the help of Baker-Deddens-Ulmann factorization Theorem \cite{bakerdeddens} we characterize the entire functions $f$ that induce an analytic Toeplitz operator $M_f$ with the double commutant property. The result can also be extended for analytic functions $\varphi$ in the Thomson-Cowen class, that is, functions $\varphi\in H^\infty(\mathbb{D})$ such that for some $a\in\mathbb{D}$ the inner part of $\varphi(z)-\varphi(a)$ is a finite Blaschke product. The Thomson-Cowen class includes the non-constant elements in $H^\infty(\overline{\mathbb{D}})$. 

As a by-product, we can obtain a geometric condition on $\varphi$ for the operator $M_\varphi$ to have the double commutant property. Specifically,
if  $\gamma$ denotes the counterclockwise unit circle $\{e^{it}\,\,:\,\,0\leq t\leq 2\pi\}$ then for any $a\in \mathbb{D}$, such that $\varphi(a)\notin \varphi(\partial \mathbb{D})$  the winding number $n(\varphi(\gamma),\varphi(a))$ must be constant.
 
The paper is structured as follows. In Section \ref{secciondos} we show that an analytic Toeplitz operator $M_\varphi$ has the minimal commutant property if and only if $\varphi$ is a weak star generator of $H^\infty(\mathbb{D})$. 
As a first consequence, we will characterize the univalent symbols $\varphi$ so that $M_\varphi$ has the double commutant property. 
We present these results by recalling some classical ideas and motivating the reader with concrete examples when possible.

In Section \ref{secciontres} we focus our attention on symbols $\varphi\in H^\infty(\mathbb{D})$ that single cover an open neighborhood and on a classical result by Deddens and Wong (we include a new proof that uses only function theory) that determines the commutant of $M_\varphi$. These maps $\varphi$ will give us the first examples of operators $M_\varphi$ that do not have the double commutant property.

In Section \ref{seccioncuatro} we discover the first examples of non-univalent functions $\varphi\in H^\infty(\mathbb{D})$ that induce $M_\varphi$ with the double commutant property. For example, when $\varphi$ is an inner function. Although these results follow from a more general result by Turner that ensures that every non-unitary isometry has the double commutant property, we will show an alternative proof.
Finally, in Section \ref{seccioncinco}, we characterize when an entire function or a function in the Thomson-Cowen class has the double commutant property. The paper closes with a brief section with concluding remarks and open questions.

\section{Analytic Toeplitz operators with a minimal commutant}
\label{secciondos}

Pioneering results on commutants of Toeplitz operators date back to the work of Shields and Wallen \cite{shieldswallen}. They consider commutants of operators that can be seen as multiplication by $z$ in a Hilbert space of analytic functions.
In fact, Deddens and Wong \cite{deddenswong} later noticed that their method can be used to show that if $\varphi\in H^\infty(\mathbb{D})$ is univalent, then
$\{M_\varphi\}'=\{ M_h\,:\,h\in H^\infty(\mathbb{D}\}$. That is, if $\varphi\in H^\infty(\mathbb{D})$ is univalent, then $\{M_\varphi\}'=\{M_z\}'$. Moreover, they show that if $T\in \{M_\varphi\}'$ then $T$ can be approximated in the weak operator topology by polynomials in $M_z$.

However, in general, we cannot assert that if $\varphi$ is univalent, then each $T\in\{M_\varphi\}'$ can be approximated in the weak operator topology by polynomials in $M_\varphi$, that is, we cannot assert that if $\varphi$ is univalent, then $M_\varphi$ has a minimal commutant. To assert minimality in the commutant of $M_\varphi$, we need an additional hypothesis. Specifically, the polynomials on $\varphi$ must be weak star dense in $H^\infty(\mathbb{D})$. In fact, we will show that $M_\varphi$ has a minimal commutant if and only if the polynomials on $\varphi$ are weak star dense in $H^\infty(\mathbb{D})$.

As a by-product of the {\it Toeplitz minimal commutant Theorem} we characterize when a Toeplitz operator $M_\varphi$ with an univalent symbol $\varphi$ has the double commutant property.

The univalence is very important in this circle of ideas. For instance, in this theorem due to Shields and Wallen, the univalence is used in part (a).  We include a proof for the sake of completeness. Let us denote by $k_a(z)$ the Hardy reproducing kernel.

\begin{theorem}[Shields-Wallen]
\label{univalent}
Assume that $\varphi\in H^\infty(\mathbb{D})$ is univalent.
\begin{enumerate}
\item[(a)] If $T\in \{M_\varphi\}'$ then there exists $h\in H^\infty(\mathbb{D})$ such that $T=M_h$. 
\item[(b)] For each $T\in \{M_\varphi\}'$ there exists a sequence of polynomials $p_n$ such that $p_n(M_z)\to T$ in the weak operator topology.
\end{enumerate}
\end{theorem}
\begin{proof}
Indeed, if $T\in \{M_\varphi\}'$, then $T^\star M_\varphi^\star=M_\varphi^\star T^\star$.
Observing that $M_\varphi^\star k_a=\overline{\varphi(a)}k_a$ we get
$$
M_\varphi^\star T^\star k_a=\overline{\varphi(a)}T^\star k_a.
$$
Since $\textrm{ker}(M_\varphi^\star-\overline{\varphi(a)} I )$ is one-dimensional, we find that there exists some complex number $\overline{h(a)}$ such that $T^\star k_a=\overline{h(a)}k_a$ for each $a\in \mathbb{D}$.

Since the function $a\in \mathbb{D}\to k_a\in H^2(\mathbb{D})$ is analytic on $\mathbb{D}$, we find that $h$ is also analytic on $\mathbb{D}$.  Moreover, since $T^\star$ is bounded, we get $h\in H^\infty(\mathbb{D})$. 
 
Finally, we consider the analytic Toeplitz operator $M_h$, and we see that $T^\star k_a=M_h^\star k_a$ for all $a\in \mathbb{D}$. Since $\{k_a(z)\,:\,a\in\mathbb{D}\}$ is a total set in $H^2(\mathbb{D})$, we obtain $T=M_h$ which proves (a).

Set $T\in \{M_\varphi\}'$, since $\varphi$ is univalent, there exists $h\in H^\infty(\mathbb{D})$ such that $T=M_h$.
For $h\in H^\infty(\mathbb{D})$ we consider $(p_n(z))$ the Féj\`er polynomials of $h$. Hence $p_n(z)\to h(z)$ uniformly on compact subsets of $\mathbb{D}$ and $\|p_n\|_\infty \leq \|h\|_\infty$, $n\geq 1$. Therefore, for each $a\in \mathbb{D}$ we obtain:
$$
\langle p_n(M_z)f, k_a\rangle=p_n(a) f(a)\to h(a) f(a) =\langle M_h f,k_a\rangle.
$$
The above equality is also true for linear finite combinations of the set $D=\{k_a\,:\,a\in\mathbb{D}\}$. Since $D$ is a total set, given $g\in H^2(\mathbb{D})$, there exists a sequence $g_l$ of linear combinations of the set $D$ such that $\|g-g_l\|_2\to 0$.

We have  that $\|p_n(M_z)\|=\|p_n\|_\infty\leq \|h\|_\infty$, therefore:
\begin{eqnarray*}
\lim_n \langle p_n(M_z)f,g\rangle=\lim_n\lim_l\langle p_n(M_z)f,g_l\rangle= \lim_l \langle M_h f,g_l\rangle= \langle M_hf,g\rangle,
\end{eqnarray*}
that is, $p_n(M_z)\to M_h$ in the weak operator topology, as we desired.
\end{proof}

We wish to characterize the analytic Toeplitz operators $M_\varphi$ with a minimal commutant. The next results show that univalence in $\varphi$ is a necessary condition.

\begin{proposition}
\label{prop1}
If $M_\varphi$ has a minimal commutant then $\varphi$ must be univalent.
\end{proposition}
\begin{proof} Indeed, if $\varphi$ is not univalent, there exists two different points in $\mathbb{D}$, $a$ and $b$, such that $\varphi(a)=\varphi(b)=c$. 

Recall that an operator $T$ has minimal commutant if and only if $T^\star$ has minimal commutant.  We will conclude if we prove that $M_\varphi^\star$ does not have a minimal commutant. To show that, we will see that $M_z^\star\notin \overline{\textrm{Alg}(M_\varphi^\star)}^{WOT}$.

Firstly, let us observe that $$\overline{\textrm{Alg}(M_\varphi^\star)}^{WOT}=\overline{\textrm{Alg}(M_\varphi^\star-cI)}^{WOT}=\overline{\textrm{Alg}(M_{\varphi-c}^\star)}^{WOT}.$$ In particular, for $c=\varphi(a)$. Set $\phi=\varphi-\varphi(a)$.

By way of contradiction, let us assume that there exists a net $p_d(M_\phi^\star)$, $d\in D$ such that
$p_d(M_\phi^\star)\rightarrow M_{z}^\star$. Let us denote by $K_a$ the reproducing kernel at $a$, normalized. Since $M_\phi^\star (K_a)=M_\phi^\star (K_b)=0$ then
$$
\lim_{d\in D} \langle p_d(M_\phi^\star)K_a,K_a\rangle=\lim_{d\in D} p_d(0)=\langle M_z^\star K_a,K_a \rangle=\overline{a}\cdot \|K_a\|^2=\overline{a}.
$$
On the other hand 
$$
\lim_{d\in D} \langle p_d(M_\phi^\star)K_b,K_b\rangle=\lim_{d\in D} p_d(0)=\langle M_z^\star K_b,K_b \rangle=\overline{b},
$$
that is $\lim_dp_d(0)=\overline{a}\neq \overline{b}=\lim_dp_d(0)$, a contradiction.
\end{proof}

The space $H^\infty(\mathbb{D})$ is a closed subspace of $L^\infty(\partial \mathbb{D})$. Thus, $H^\infty(\mathbb{D})$ can be seen as the dual of a quotient of $L^1(\partial \mathbb{D})$.
Since $H^\infty(\mathbb{D})$ is the dual of a separable Banach space,  it is possible to define a weak star topology in $H^\infty(\mathbb{D})$ (see \cite{sarason}). Let us denote this topology by $\tau_{\star}$.

A function $\varphi\in H^\infty(\mathbb{D})$ is said to be a weak star generator of $H^\infty(\mathbb{D})$ provided the polynomials on $\varphi$ are weak star dense in $H^\infty(\mathbb{D})$.

The main result of this section involves some technicalities when dealing with topologies defined on $H^\infty(\mathbb{D})$.
Assume that $\varphi$ is univalent. Applying Proposition \ref{prop1}, if $M_{h_\alpha}$ is a net converging to an operator $T$ in the weak operator topology, then $T=M_h$ for some $h\in H^\infty(\mathbb{D})$.  In $H^\infty(\mathbb{D})$, we can also consider the relative topology induced by the weak operator topology. If $g\in H^\infty(\mathbb{D})$ a basic neighborhood of $g$ is defined as:
$$
V_{wot}(g;f_1,f_2;\varepsilon)=\{h\in H^\infty(\mathbb{D})\, :\, |\langle (M_h-M_g)f_1,f_2\rangle|<\varepsilon\}
$$
for some $f_1,f_2\in H^2(\mathbb{D})$. Let us denote this topology by $\tau_{wot}$.

Moreover, if $\varphi$ is univalent, then by Proposition \ref{prop1} the WOT closure of the polynomials on $M_{\varphi}$ is a subspace of $\{M_z\}'$. Thus, it can be identified as a subspace of $H^\infty(\mathbb{D})$, which is exactly the $\tau_{wot}$-closure of the polynomials on $\varphi$.

Analogously, we can make a similar observation regarding a stronger topology on $\mathcal{L}(H^2(\mathbb{D}))$: the $\sigma$-weak star (or ultra-weak star) operator topology.
Let us denote this topology by $\tau_{\sigma *}$. If $g\in H^\infty(\mathbb{D})$ then a basic neighborhood of $g$ on $\tau_{\sigma*}$ is defined as:
$$
V_{\sigma *}(g;(f_n)_n,(g_n)_n;\varepsilon)=\left\{h\in H^\infty\,:\, \left|\sum_{n=1}^\infty \langle (M_g-M_h)f_n,g_n \rangle \right|<\varepsilon\right\}
$$
when $f_n,g_n\in H^2(\mathbb{D})$ and $\sum_{n=1}^\infty \|f_n\|^2<\infty$ and $\sum_{n=1}^\infty \|g_n\|^2<\infty$ (here $\|\cdot\|$ denotes the norm on $H^2(\mathbb{D})$). 
We refer the reader to Takesaki's book \cite{takesaki} where these operator topologies are described.

\begin{lemma}
\label{topologies}
The weak operator topology and the $\sigma$-weak star operator topology restricted to $H^\infty(\mathbb{D})$ coincide.
\end{lemma}
\begin{proof}
In general, the $\sigma$-weak star operator topology is stronger than the weak operator topology. Hence,
it is sufficient to show that each $\tau_{\sigma *}$ basic neighborhood of $0$  is a $\tau_{wot}$ neighborhood of $0$. A basic $\tau_{\sigma *}$ neighborhood of $0$ is defined as:
$$
V_{\sigma *}(0;(f_n),(g_n);\varepsilon)=\left\{h\in H^\infty \,:\, \left|\sum_{n=1}^\infty\langle hf_n,g_n\rangle \right|<\varepsilon\right\}.
$$
where $\sum_{n=1}^\infty \|f_n\|^2<\infty$ and $\sum_{n=1}^\infty\|g_n\|^2<\infty$.

If $\|\cdot\|_1$ denotes the standard norm in $H^1(\mathbb{D})$ then by applying Holder inequalities we get:
\begin{eqnarray*}
\left\|\sum_{n=1}^\infty f_n \overline{g_n} \right\|_1 &\leq & \sum_{n=1}^\infty \|f_n\overline{g_n}\|_1 \\
&\leq & \sum_{n=1}^\infty \|f_n\| \|g_n\| \\
&\leq & \left(\sum_{n=1}^\infty \|f_n\|^2\right)^{1/2}  \left(\sum_{n=1}^\infty \|g_n\|^2\right)^{1/2} <\infty.
\end{eqnarray*}
This implies that $F(z)=\sum_{n=1}^\infty f_n\overline{g_n}\in H^1(\mathbb{D})$ and by Fubini's Theorem we can exchange the sum and the integral:
$$
\left|\sum_{n=1}^\infty\langle hf_n,g_n\rangle \right|=\left|\int_{\partial\mathbb{D}} h(z)F(z)\,dz\right|.
$$
Le us denote by $B$ a Blaschke product with the same zeros of $F$. Since $g(z)=F(z)/B(z)$ has no zeros, we can write $F(z)=B(z) g(z)^{1/2} g(z)^{1/2}$. And since $F\in H^1(\mathbb{D})$,
we see that $B(z)g(z)^{1/2}$ and $g(z)^{1/2}$ are functions of $H^{2}(\mathbb{D})$. Therefore: 
$$
\left|\int_{\partial\mathbb{D}} h(z)F(z)\,dz\right|=\left|\langle h(z)B(z)g(z)^{1/2},\overline{g}^{1/2}\rangle\right|.
$$
Hence, 
$$
V_{\sigma *}(0,(f_n),(g_n);\varepsilon)=V_{wot}(0; B g^{1/2}, \overline{g}^{1/2}; \varepsilon)
$$
which yields the desired result.
\end{proof}

\begin{lemma}
\label{trans}
Assume that $h\in H^\infty(\mathbb{D})$. The following conditions are equivalent:
\begin{enumerate}
\item[a)]  $h\in H^\infty(\mathbb{D})$ is weak star limit of  $p_n(\varphi)$. 
\item[b)]  $h\in H^\infty(\mathbb{D})$ is the WOT limit of $p_n(M_\varphi)$.
\item[c)] $h\in H^\infty(\mathbb{D})$ is the $\sigma$ weak star limit of $p_n(M_\varphi)$.

\end{enumerate}

\end{lemma}
\begin{proof}
The equivalence between b) and c) follows from the Lemma \ref{topologies}.
To establish a) implies b), assume that $h$ is a weak star limit of polynomials on $\varphi$, then (\cite[Lemma 1]{sarason}) there exists a sequence of polynomials $p_n$, such that, $p_n(\varphi)(a)\to h(a)$ for any $a\in \mathbb{D}$ and $\|p_n(\varphi)\|_\infty<M$ for all $n$.
Given $f,g\in H^2(\mathbb{D})$ and $\varepsilon>0$, we wish to prove that there exists $n_0$ such that
$|\langle (M_{p_n(\varphi)}-M_h)f,g\rangle|<\varepsilon $ for all $n\geq n_0$.

 Since the set of reproducing kernels $\{k_a(z)\, :\,a\in \mathbb{D}\}$ is a total set in $H^2(\mathbb{D})$, there exists a finite linear combinations of reproducing kernels $g_\varepsilon$ such that 
\begin{equation}
\label{desigualdad1}
\|g-g_{\varepsilon}\|\leq \frac{\varepsilon}{2\|f\|(M+\|h\|_\infty)}.
\end{equation}
The pointwise convergence of the hypothesis implies that $\langle (M_{p_n(\varphi)}-M_h)f,k_a \rangle$ converges to zero. Therefore, there exists $n_0$ such that for $n\geq n_0$
\begin{equation}
\label{desigualdad2}
|\langle (M_{p_n(\varphi)}-M_h)f,g_\varepsilon \rangle|<\varepsilon/2.
\end{equation}
Finally, by using (\ref{desigualdad1}) and (\ref{desigualdad2}) and the triangular inequality, we get that for any $n\geq n_0$
\begin{eqnarray*}
|\langle (M_{p_n(\varphi)}-M_h)f,g\rangle| &\leq & |\langle (M_{p_n(\varphi)}-M_h)f,g_\varepsilon \rangle|+|\langle (M_{p_n(\varphi)}-M_h)f,g-g_\varepsilon \rangle|\\
&\leq & \varepsilon/2+ \frac{\varepsilon}{2} \frac{\|p_n(\varphi)-h\|_\infty}{M+\|h\|_\infty}  \leq \varepsilon.
\end{eqnarray*}

Conversely, if $M_{p_n(\varphi)}\to M_h$ in WOT then $\langle p_n(\varphi), k_a\rangle=p_n(\varphi(a))\to h(a)=\langle h,k_a\rangle$.
On the other hand, if $\langle M_{p_n(\varphi)}f, g\rangle\to \langle M_hf,g\rangle$ for any $f,g\in H^2(\mathbb{D})$ then $M_{p_n(\varphi)}f\to M_hf$ weakly. Thus, $M_{p_n(\varphi)}f$ is pointwise bounded for any $f\in H^2(\mathbb{D})$. Therefore, by the uniform boundedness principle we get $\|M_{p_n(\varphi)}\|=\|p_n(\varphi)\|_\infty<M$ for all $n$. That is, $p_n(\varphi)\to h$ in the weak star topology, which establishes the equivalence between a) and b).
\end{proof}
\begin{remark}
\label{trans2}
The proof of the above result gives us a little more. Specifically, on $H^\infty(\mathbb{D})$, the three topologies have the same convergent sequences.
That is,  $\psi_n\in H^\infty(\mathbb{D})$ converges weak star to $h$ if and only if $M_{\psi_n}$ converges in the weak operator topology to $M_h$.
\end{remark}

To show Theorem \ref{main1} we will use the sequential approach of the weak star topology introduced by Banach and Mazurkiewicz (\cite[Annexe]{banach}).
If $\mathcal{P}$ is a linear subspace of a Banach space $X$ with separable predual, then the weak star closure of $\mathcal{P}$ can be recovered by taking limits on $\mathcal{P}$ in a transfinite way. Specifically, if $B$ is the unit ball of $X$,
the derived set of $\mathcal{P}$ is defined as:
$$
\mathcal{P}^{(1)}=\bigcup_{n=1}^\infty \overline{\mathcal{P}\cap n B}^{*},
$$
which represents the weak star limit of sequences in $\mathcal{P}$. In general, it is well known that $\mathcal{P}^{(1)}$ does not coincide with the weak star closure of $\mathcal{P}$. In this way, it is natural to introduce derived sets for any ordinal number as follows: if $\mathcal{P}^{(\alpha)}$ has already been defined, then $\mathcal{P}^{(\alpha+1)}=(\mathcal{P}^{(\alpha)})^{(1)}$.

When the predual of $X$ is separable, then there exists an ordinal $\beta$ such that $\mathcal{P}^{(\beta)}=\mathcal{P}^{(\beta+1)}=\overline{\mathcal{P}}^{*}$. That is, we can recover the weak star closure of a subspace by taking derived sets in a transfinite way. Now we are in a position to show the main result of this section.

\begin{theorem}
\label{main1}
Set $\varphi\in H^\infty(\mathbb{D})$. $M_\varphi$ has a minimal commutant if and only if $\varphi$ is a weak star generator of $H^\infty(\mathbb{D})$.
\end{theorem}
\begin{proof}
Firstly, let us observe that if $M_\varphi$ has a minimal commutant then $\varphi$ is univalent (see Proposition \ref{prop1}). On the other hand, if $\varphi$ is a weak star generator of $H^\infty(\mathbb{D})$, then also $\varphi$ must be univalent (see \cite[Proposition 3]{sarasondos}). Thus, we can assume without loss that $\varphi$ is univalent.

Let us denote by $\mathcal{P}$ the linear subspace of the polynomials on $\varphi$.
We claim that the weak star closure of $\mathcal{P}$: $\overline{\mathcal{P}}^{*}$ and the WOT closure of $\mathcal{P}$:  $\overline{\mathcal{P}}^{WOT}$ are the same set.

Indeed, let us denote by $\mathcal{P}_{*}^{(\alpha)},\mathcal{P}_{\sigma*}^{(\alpha)},$ and $\mathcal{P}_{wot}^{(\alpha)}$ the derived set of order $\alpha$ in the weak star, $\sigma$ weak star and WOT topologies, respectively. By applying Lemma \ref{trans} and Remark \ref{trans2}, we get that 
for any ordinal $\alpha$ we obtain $$\mathcal{P}_{*}^{(\alpha)}=\mathcal{P}_{\sigma*}^{(\alpha)}=\mathcal{P}_{wot}^{(\alpha)}.$$

Since the predual of $H^\infty$ (a quotient of $L^1(\mathbb{D})$) and the predual of $\mathcal{L}(H^2(\mathbb{D}))$ (the trace class) are separable, there exists an ordinal $\beta$ such that
$$
\overline{\mathcal{P}}^{\sigma*}=\mathcal{P}_{\sigma*}^{(\beta)}=\mathcal{P}_{*}^{(\beta)}=\overline{\mathcal{P}}^{*}.
$$
Finally, by applying Lemma \ref{topologies} we get $\overline{\mathcal{P}}^{\sigma*}=\overline{\mathcal{P}}^{wot}$.
Now, the result follows as an observation. If $\varphi$ is univalent, then by Shields-Wallen the commutant of $M_\varphi$ is $H^\infty(\mathbb{D})$. Thus, the WOT closure of the polynomials on $\varphi$ is $H^\infty(\mathbb{D})$ if and only if the weak star closure of polynomials on $\varphi$ is $H^\infty(\mathbb{D})$, that is, $\varphi$ is a weak star generator of $H^\infty(\mathbb{D})$. \end{proof}

\begin{remark}
\label{refesub}
For future reference, from the proof of the above result, we can deduce that if $\mathcal{M}\subset H^\infty(\mathbb{D})$ is a linear subspace, then $\overline{\mathcal{M}}^{*}=\overline{\mathcal{M}}^{WOT}$.
\end{remark}

The next result, with the help of Theorem \ref{main1}, characterizes those univalent symbols $\varphi$ for which the analytic Toeplitz operator $M_\varphi$ has the double commutant property.

\begin{corollary}
\label{corolario}
Assume that $\varphi\in H^\infty(\mathbb{D})$ is univalent. Then $M_\varphi$ has the double commutant property if and only if the polynomials on $\varphi$ are weak star dense in $H^\infty(\mathbb{D})$.
\end{corollary}
\begin{proof} Indeed, clearly if the polynomials on $\varphi$ are weak star dense in $H^\infty(\mathbb{D})$ then $M_\varphi$ has a minimal commutant. Therefore $M_\varphi$ has trivially the double commutant property.

Conversely, if $\varphi$ is univalent, then $\{M_\varphi\}'=\{M_z\}'$, therefore $\{M_\varphi\}''=\{M_z\}''=\{M_h\,:\, h\in H^\infty(\mathbb{D})\}$. Therefore, since $M_\varphi$ has the double commutant property then
$$
\overline{\textrm{Alg}(M_\varphi)}^{WOT}=\{M_\varphi\}''=\{M_z\}'=\{M_\varphi\}'.
$$
Thus $M_\varphi$ has a minimal commutant, therefore $\varphi$ is a weak star generator of $H^\infty(\mathbb{D})$ as we desired.
\end{proof}


We note that the minimal commutant property for an analytic Toeplitz operator $M_\varphi$ depends on the shape of $G=\varphi(\mathbb{D})$. For instance, using several classical results on approximation we can find several examples of univalent maps $\varphi$ which are generators of $H^\infty(\mathbb{D})$. For example, if $\varphi(z)=z$ then the Féjèr polynomials are weak star dense on $H^\infty(\mathbb{D})$. If $\varphi$ maps $\mathbb{D}$ univalently onto $G$ and $G$ is a simply connected domain whose boundary is a Jordan curve, then Walsh's Theorem asserts that the polynomials on $\varphi$ are weak star dense on $H^\infty(\mathbb{D})$. 

As D. Sarason \cite{sarason} pointed out, the domains for which $\varphi$ is a weak star generator of $H^\infty$ or order 1, were characterized by Farrel in \cite{farrel1,farrel2}. That is, $\varphi$ is a weak star generator of $H^\infty(\mathbb{D})$ if and only if $\varphi(\mathbb{D})$ is a Caratheodory domain. In \cite{sarason} Sarason provided examples of maps $\varphi$ which are weak star generators of $H^\infty(\mathbb{D})$ of order (finite) different from $1$.

On the other hand, it is easy to find maps that are not generators of $H^\infty(\mathbb{D})$. By \cite[Proposition 1]{sarasondos}, if $\varphi$ is a generator of $H^\infty(\mathbb{D})$, then $\varphi$ must be univalent almost everywhere in $\partial \mathbb{D}$. Thus, if $\varphi$ maps $\mathbb{D}$ univalently onto the slit disk $\mathbb{D}\setminus [0,1)$, then $M_\varphi$ do not have a minimal commutant, and also by applying Corollary \ref{corolario} $M_\varphi$ does not have the double commutant property.

Moreover, it is possible to find maps that are univalent in the unit circle except two points which do not induce analytic Toeplitz operators with a minimal commutant. For example, let us consider the region formed by the disk with center zero and radius one, minus the closure of the disk with center $1/2$ and radius $1/2$. In this special moom-shaped domain, it is not possible to approximate by polynomials (see \cite{gaier} p. 22).  If we consider $\varphi$ an analytic function that maps $\mathbb{D}$ to $G$, then there are bounded analytic functions in $G$ that cannot be approximated by polynomials; therefore, the operator $M_\varphi$ does not have a minimal commutant.

\begin{figure}[ht!]
\centering
\includegraphics[scale=1]{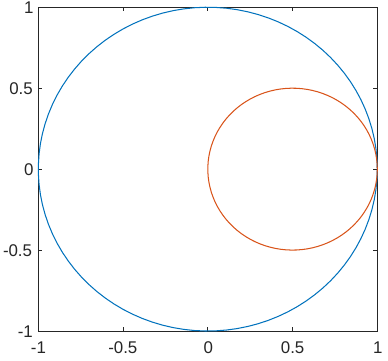}
\caption{The lunar eclipse}
\end{figure}

Given $\varphi\in 
H^\infty(\mathbb{D})$, since the
commutant $\{M_{\varphi}\}'=
\{M_{\varphi/R}\}'$, we can suppose without loss of generality that
$\|\varphi\|_\infty<1$.

From the results of Deddens
(\cite{deddens1,deddens2})
there exists an important connection between analytic Toeplitz operators
and composition operators.

The results obtaining here have
similar flavor to the results relating
to the cyclicity and density
of the range of a composition operator.
For example, the hypothesis of univalence on the
map is necessary (even the hypothesis of univalence
almost everywhere on the unit circle
$\partial\mathbb{D}$). However, there
is no equivalence between the cyclicity
of $C_\varphi$ and the minimal
commutant property of the operator
$M_{\varphi}$. Indeed, as
pointed out in
\cite{memoirsBourdonShapiro} the
composition operator $C_\varphi$
induced by $z/(2-z)$ is not cyclic,
however $M_\varphi$ has a minimal
commutant.

The range of $C_\varphi$ is dense means
exactly that the polynomials on
$\varphi$ are dense in
$H^2(\mathbb{D})$. If $\varphi$ is a weak star generator of $H^\infty(\mathbb{D})$, then the range of $C_\varphi$ is dense (see also \cite{roan}). However, the converse is not true, there are maps $\varphi$ such that the polynomials on $\varphi$ are dense in $H^2(\mathbb{D})$ but $\varphi$ is not a weak star generator of $H^\infty(\mathbb{D})$ (see \cite{akeroyd}).


\section{Double commutant property for analytic Toeplitz operators induced by maps that single cover an open neighborhood}
\label{secciontres}
Maps very close to univalent ones are those $\varphi$ nonunivalent for which $\varphi$ single covers an open neighborhood $W$ of $\varphi(\mathbb{D})$.
We say that \textit{$\varphi$ singly covers a set $W$} whenever $\varphi$ takes some set $V \subset \mathbb{D}$ univalently onto $W$ with $\varphi^{-1}(W)=V$, that is, the points of $V$ are the only ones in $\mathbb{D}$ taken by $\varphi$ into $W$.

Let us denote by $S\subset H^\infty(\mathbb{D})$ the set of such maps. The objective of this section is to prove that if $\varphi\in S$ then $M_\varphi$ does not have the double commutant property.

Given $h\in S$, Deddens and Wong (\cite{deddenswong}) proved that $\{M_h\}'=\{M_z\}'$. They obtained such a result as a byproduct of a characterization of the commutant of an analytic Toeplitz operator $M_\varphi$ in terms of the commutants of the symbols of their inner-outer factorization. This result was later obtained by Bourdon and Shapiro (\cite{bourdonshapiro}) as a byproduct of a deep result on intertwining
analytic Toeplitz operators, which is intimately connected on a still open conjecture posed by Deddens.

In fact, such a result can be obtained using the ideas of Shields and Wallen used in Theorem \ref{univalent}.
For the sake of completeness, we include a new proof of this fact that uses only function theory.

\begin{theorem}[Deddens-Wong]
\label{deddeswong}
Assume that $\varphi$ single covers a nonempty neighborhood $W$ of $\varphi(\mathbb{D})$, then $\{M_\varphi\}'=\{M_z\}'$.
\end{theorem}
\begin{proof} To see the ideas clearly, let's divide the proof into several steps.

{\it Steep 1. 
Let $\varphi \in H^\infty(\mathbb{D})$, continuous in $\overline{\mathbb{D}}$. If for some 
$a\in \mathbb{D}$, $\varphi(z)-\varphi(a)$ has only a simple zero in $\mathbb{D}$ and $X$ commutes with $M_\varphi$ then $(Xf)(a)=(X1)(a) f(a)$.}

Note that $\varphi(z)-\varphi(a)=(z-a)g(z)$, with $g$ continuous and non-vanishing
in $\overline{\mathbb{D}}$. By the maximum principle $1/g \in H^\infty(\mathbb{D})$. 

On the other hand, if $h\in H^2(\mathbb{D})$, and $h(a)=0$, then $h(z)=(z-a)u(z)$ for
some $u\in H^2(\mathbb{D})$. Therefore,
$$
h(z)=\varphi(z)-\varphi(a)=u(z)\frac{u(z)}{g(z)}.
$$

Next, let $f\in H^2(\mathbb{D})$, then $f(z)-f(a)$ vanishes at $z=a$. Hence, it can be written as $f(z)-f(a)=(\varphi(z)-\varphi(a))F(z)$ for some $F\in H^2(\mathbb{D})$.

Since $X$ commutes with $M_\varphi$, we get
$$X(f-f(a))(z) = X(M_\varphi F-\varphi(a)F)(z)=\varphi(z)XF(z)-\varphi(a)XF(z).$$
Therefore, $0=X(f-f(a))(a)=Xf(a)- f(a)X1(a)$, which proves the first steep.

{\it Steep 2.  Let $\varphi\in H^\infty(\mathbb{D})$, such that $\varphi$ single covers a non-empty open subset $W\subset \varphi(\mathbb{D})$. If $X$ commutes with $M_\varphi$ then $Xf(a)=(X1)(a)f(a)$.}
The proof is the same as in Steep 1 as long we can write $\varphi(z)-\varphi(a)=(z-a)g(z)$; $z\in \mathbb{D}$ with $g$ and $1/g\in H^\infty(\mathbb{D})$.

Let $V=\varphi^{-1}(W)$, thus, $\varphi$ is univalent in $V$. Therefore, if $a\in V$, $\varphi'(a)\neq 0$ and we can write:
$$
\varphi(z)-\varphi(a)=(z-a)g(z),\,\,z\in \mathbb{D}.
$$
Since $\varphi\in H^\infty(\mathbb{D})$, $g$ is bounded in $\mathbb{D}$. Moreover, by our hypothesis $\varphi(z)-\varphi(a)$ only vanish in $\mathbb{D}$ at the point $a$. Thus $g(z)\neq 0$ for all $z\in \mathbb{D}$. We claim that for all $z\in \mathbb{D}$, $|g(z)|>c$ for some $c>0$.

Let us consider a ball $B(a,r)$ centered at $a$ of radius $r>0$ with $\overline{B(a,r)}\subset V$. Then $\varphi^{-1}(\varphi(B(a,r))=B(a,r)$. Assume that $g$ is not bounded below in $\mathbb{D}$. Then, there is a sequence $(z_n)\subset \mathbb{D}$ with $g(z_n)\to 0$. Since $g(z)\neq 0$ for all $z\in \mathbb{D}$, $(z_n)$ has to approach $\partial \mathbb{D}$. Thus, by choosing $n$ large enough, $z_n\notin B(a,r)$ and $\varphi(z_n)\in \phi(B(a,r))$ which is a contradiction.

{\it Steep 3.} Finally, let us show that $\{M_\varphi\}'=\{M_z\}'$. Since $M_z$ has a minimal commutant then $\{M_z\}'\subset \{M_\varphi\}'$. Set $X\in \{M_\varphi\}'$. Since the function $a\in \mathbb{D}\to k_a(z)$ is analytic we get that for each $f\in H^2(\mathbb{D})$ the function  
$$
\langle Xf,k_a\rangle=h_f(a)
$$
is analytic on $\mathbb{D}$.
Moreover, since $X$ is bounded by the 
Cauchy-Schwarz  inequality we get that $h_f$ is bounded on $\mathbb{D}$. By {\it steep 2} we get  that for any $a\in W$
$
h_f(a)=\langle Xf,a\rangle=(Xf)(a)= (X1)(a) f(a).
$
In particular for $f=1$, we obtain that $h_1(z)=(X1)(z)\in H^\infty(\mathbb{D})$.
Hence for all $a\in \mathbb{D}$, $Xf(a)=(X1)(a)f(a)$, that is $X$ is an analytic Toeplitz operator with symbol $(X1)(z)$ and therefore $X\in \{M_z\}'$ as we desired.
\end{proof}

Let us show that the elements of class $S$ do not have the double commutant property.

\begin{proposition}
\label{singlecover}
Assume that $\varphi\in S$, that is, $\varphi$ is not univalent and $\varphi$ single covers a non-empty neighborhood $W$ of $\varphi(\mathbb{D})$. Then $M_\varphi$ has not the double commutant property.
\end{proposition}
\begin{proof}
Since $\varphi$ is not univalent, then the polynomials on $\varphi$ are not dense in $H^2(\mathbb{D})$. Thus, there exists $f_0\in H^2(\mathbb{D})\setminus \{0\}$ such that
$$
\langle p(M_\varphi)1,f_0\rangle=0
$$
for each polynomial $p$.
Consider the linear functional $\gimel: \mathcal{L}(H^2(\mathbb{D}))\to \mathbb{C}$ defined by:
$$
\gimel(X)=\langle X1,f_0\rangle.
$$
The density of $H^\infty(\mathbb{D})$ in $H^2(\mathbb{D})$ and the fact that $f_0\neq 0$, implies the existence of $h\in H^\infty(\mathbb{D})$ such that $\langle h,f_0\rangle\neq 0$. Since $\varphi\in S$,  Deddens-Wong's Theorem implies that $M_h$ commutes with each element in $\{M_\varphi\}'=\{M_z\}'$, that is, $M_h\in\{M_\varphi\}''$.

By continuity we get that $\gimel$ is a linear functional such that $\gimel(Y)=0$ for all $Y\in \overline{\textrm{Alg}(M_\varphi))}^{WOT}$
and $\gimel(M_h)=\langle h,f_0\rangle\neq 0$. Thus, $M_\varphi$ do not have the double commute property, as we wanted to prove.
\end{proof}

Let us point out that T. Turner in his 1971 Michigan dissertation (see \cite[Appendix D p.91]{turner}) posed the following question: Assume that $T$ has the double commutant property. Does $T^n$ have the double commutant property for all $n$? This question was solved by Deddens and Wogen in \cite[Examples 1,2 p. 362]{deddenswogen}.  Proposition \ref{singlecover} helps us to find some more natural counterexamples.  

\begin{example}
If we consider the cardioid map $\varphi(z)=(z+1/2)^2$,
in Figure \ref{fig1} the regions inside the inner loop are double covered by $\varphi$. On the other hand, the cardioid region, outside of the small loop, is covered in a single way by $\varphi$. Therefore, $M_\varphi$  does not have the double commutant property. 

Thus, if we consider the translation $\psi(z)=z+1/2$,  by Walsh's Theorem we see that $M_\psi$ has the minimal commutant property, therefore $M_\psi$ trivially has the double commutant property, but
$M_\psi^2=M_\varphi$ does not.

\end{example}

\begin{figure}[h!]
\centering
\includegraphics[scale=1]{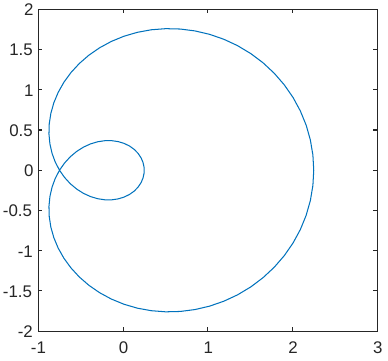}
\caption{The map $\varphi(z)=(z+1/2)^2$}
\label{fig1}
\end{figure}

\section{The double commutant of $M_{z^n}$.}
\label{seccioncuatro}

 The commutant of $M_{z^n}$, $n\geq 2$ was studied by Z. \v{C}u\v{c}kovi\'{c} in \cite{cucko}. The commutant of $M_{z^n}$, $n\geq 2$ is very big. \v{C}u\v{c}kovi\'{c} was able to describe the elements in the algebra of all Toeplitz operators that commutes with $M_{z^n}$. The commutant of $M_{z^n}$, $n\geq 2$ is so big, that the double commutant of $M_{z^n}$, $n\geq 2$ should be minimal.

In fact, this result is a special case of a general result of Turner \cite{turner2} which asserts that if $X$ is a nonunitary isometry, then $X$ has the double commutant property. 

However, we wish to show an alternative proof of this fact for the special case $M_{z^n}$, $n\geq 2$. 
\begin{proposition}
\label{doss}
The operator $M_{z^n}$, $n\geq 2$ has the double commutant property.
\end{proposition}
\begin{proof}

Assume that $S$ is a linear operator in the double commutant of $M_{z^n}$.
 Since $S$ commutes with $M_z$ then
$S=M_h$ for some bounded analytic function $h$.

 For any analytic function $f(z)=\sum_{n=0}^\infty a_nz^n\in H^2(\mathbb{D})$, we can decompose  
$$
f(z)=f_1(z)+zf_2(z)+\cdots + z^{n-1}f_n(z)
$$
where $f_i(z)=F_i(z^n)$ for some $F_i\in H^2(\mathbb{D})$. 

We consider the operator $Tf(z)=f_1(z)$ which is clearly bounded on $H^2(\mathbb{D})$ and 
 commutes with $M_{z^n}$.
Therefore, forcing $M_h$ to commute with $T$ we get
$M_hT1=h=h_1$, therefore $h=H_1(z^n)$ for some $H_1\in H^\infty(\mathbb{D})$.

Next, the argument follows similar to the proof of Theorem \ref{univalent} by Shields-Wallen. We consider $p_n$ the Féjèr polynomials of $H_1$ and we have that $p_n(z)\to H_1(z)$ uniformly on compact subsets of $\mathbb{D}$ and $\|p_n\|_\infty\leq \|H_1\|_\infty$. 
And this fact implies that $(p_k(M_{z^n}))$ converges to $M_h=M_{H_1(z^n)}$ in the weak operator topology.
\end{proof}
To avoid repeating arguments,
we state the following lemma, whose
proof follows by the end of the proof
of Proposition \ref{doss}.

\begin{lemma}
\label{corte}
Assume that $h\in H^\infty (\mathbb{D})$ and $B$ is an inner function. If there exists $\varphi\in H^\infty(\mathbb{D})$ such that $h(z)=\varphi(B)$, then $M_h\in \overline{\textrm{Alg}(M_{B})}^{WOT}$.
\end{lemma}

\begin{remark}
The commutant of $M_{z^n}$ is so big, that there are elements that are not in the Toeplitz algebra.
For example, set $a=re^{i\theta_0}\in \mathbb{D}$ and consider
$a_k=r^{1/n}e^{\frac{\theta_0+2k\pi}
{n}i}$,
$k=0,\cdots,n-1$ the $n$-roots of $a$. 
Set $\lambda_0=e^{2\pi i/n}\in 
\partial \mathbb{D}$ the ``first" root 
of the unity. Clearly
$a_k=a_0\lambda_0^k$, for $k=0,\cdots, 
n-1$. The dilation operator
defined by $L_{\lambda_0}f(z)=f(\overline{\lambda
_0} z)$, also commutes with $M_{z^n}$. If we force $M_h$ to commute with $L_{\lambda_0}^k$, $k=1,...,n$ we find that $h$ is constant on each $z^n$ fiber, $(z^n)^{-1}$, therefore there is a unique analytic function $\varphi$
such that $h(z)=\varphi(z^n)$, and from this we find again that $M_{z^n}$ has the double commutant property.
\end{remark}

For future reference, we can deduce the following corollary.
\begin{corollary}
\label{doblen}
If $X\in \{M_{z^k}\}''$, then there exists $h\in H^\infty(\mathbb{D})$ such that $X=M_{h(z^k)}$.
\end{corollary}

For the sake of complementing the exposition, we will extend the idea of Proposition \ref{doss} to any inner function $g$.
For any isometry $V$ in a Hilbert space $H$, the Wold decomposition Theorem asserts that $H=H_0\oplus \sum_{n=1}^\infty \oplus V^n(H_1)$, where $H_1=H\ominus V(H)$ and $H_0=\bigcap_{n=1}^\infty V^n(H)$.
When $H=H^2(\mathbb{D})$ and $V$ is multiplication $M_g$ by an inner function $g$, then $H_0=\{0\}$ and 
$$
H^2(\mathbb{D})=\sum_{n=1}^\infty \oplus s_n(z) H^2[g]
$$
when $H^2[g]=\{ f(g)\, :\,f\in H^2(\mathbb{D})\}$ and $\{s_1(z),\cdots,s_n(z)\,\cdots\}$ is an orthonormal basis of $H^2(\mathbb{D})\ominus gH^2(\mathbb{D})$ that can be constructed by the Gram–Schmidt process. If $g(0)=0$ we can take $s_1=1$ and if not, we can take $s_1(z)=K_a(z)$, where $K_a$ denotes the normalized reproducing kernel at $a$, and $g(a)=0$.
\begin{proposition}
If $g$ is an inner function, then $M_g$ has the double commutant property.
\end{proposition}
\begin{proof}
If $X\in \{M_g\}''$, then $X=M_h$, $h\in H^\infty(\mathbb{D})$. It is sufficient to show that $h(z)=H(g)$ for some $H\in H^\infty(\mathbb{D})$.
If $f\in H^2(\mathbb{D})$, then $f(z)=\sum_{n=1}^\infty s_n(z)F_n(g)$ for $F_n\in H^2(\mathbb{D})$ $n\geq 1$. And $s_1(z)$ is $1$ or $K_a(z)$.
We consider the projection operator
$Tf=s_1(z)F_1(g)$ that is bounded on $H^2(\mathbb{D})$ (see \cite{lance}) and commutes with $M_g$. Therefore, forcing $M_h$ to commute with $T$ we get
$$
M_hTs_1(z)=s_1(z)h(z)=T(hs_1(z))=s_1(z)H_1(g).
$$
Therefore $h=H_1(g)$. Since $g$ is inner and $h\in H^\infty(\mathbb{D})$ we get $H_1\in H^\infty(\mathbb{D})$ as we wanted to prove.
\end{proof}

For future reference, we can deduce the following corollary.
\begin{corollary}
\label{doblen2}
Assume that $B$ is an inner function.
If $X\in \{M_{B}\}''$, then there exists $h\in H^\infty(\mathbb{D})$ such that $X=M_{h(B)}$.
\end{corollary}

\section{Double commutant property for analytic Toeplitz operators induced by entire functions and functions in the Thomson-Cowen class.}
\label{seccioncinco}
In this section, we will explore when an analytic Toeplitz operator $M_\varphi$ 
induced by an entire function $\varphi$ or a function of the Thomson-Cowen class has the double commutant property. Although we will do a unified study of both classes, to motivate the reader, we are going to look at some examples.

Our double commutant property for analytic Toeplitz operators induced by entire functions uses a result by Baker, Deddens, and Ullman \cite{bakerdeddens}.

Let us denote by $\gamma$ the unit circle $\{e^{it}\,\,:\,\,0\leq t\leq 2\pi\}$.
 Assume that $\varphi$ is analytic on $\overline{\mathbb{D}}$, if $a\notin \varphi(\gamma)$ we denote
 by $n(\varphi(\gamma),a)$ the winding number of $\varphi$ about $\gamma$ and we set
 $$
 k(\varphi)=\inf\{n(\varphi(\gamma),a)\,:\,n(\varphi(\gamma),a)\neq 0\}.
 $$
\begin{theorem}[Baker-Deddens-Ullman] If $\varphi$ is a non-constant entire function and $k=k(\varphi)$ then there exists an entire function $h$  such that $\varphi(z)=h(z^k)$  and $k(h)=1$. 
\end{theorem}
As a by-product of this result, they obtained that if
$\varphi(z)=h(z^k)$ then $\{M_\varphi\}'=\{M_{z^k}\}'$.
Let us see the following geometric necessary condition, if the points of the region $\varphi(\mathbb{D})$ have different winding numbers with respect to the curve $\varphi(\gamma)$ then $M_\varphi$ do not have the double commutant property.

\begin{proposition}
\label{geometric}
Assume that $\varphi$ is an entire function with
$k(\varphi)=k$. If there exists a point $c\in 
\varphi(\mathbb{D})$ such that 
$n(\varphi(\gamma),c)=p$ and 
 $p>k=k(\varphi)$  then $M_\varphi$ does not have the double commutant property.
\end{proposition}
\begin{proof}
By the Baker-Deddens-Ullman result, there is an open 
subset $W\subset \varphi(\mathbb{D})$ such that 
$n(\varphi(\gamma),a)=k=k(\varphi)$ for all $a\in W$, 
and $\varphi(z)=h(z^k)$ for some entire function $h$ 
which satisfies that $n(h(\gamma),a)=1$ for any $a\in 
W$.

By the Argument Principle (see \cite[p. 152]{ahlfors}), 
we obtain that $h$ single covers the open subset $W$. 
Moreover, since $n(\varphi(\gamma),c)=p
>k$ we obtain that $h$ is not univalent.
Hence, by results of Section \ref{secciontres},  we proved 
that $M_h$ do not have the double commutant property. 
Specifically, since the polynomials on $h$ are not 
dense in $H^2(\mathbb{D})$ there exists $f\in 
H^2(\mathbb{D})$ such that for any polynomial $p$, 
$\langle p(h),f\rangle=0$, and since 
$H^\infty(\mathbb{D})$ is dense in $H^2(\mathbb{D})$ 
there exists $g\in H^\infty(\mathbb{D})$ such that 
$\langle g,f\rangle\neq 0$. 

As a consequence we get that
$
\langle p(h(z^k)),f(z^k)\rangle=0
$
and $\langle g(z^k),f(z^k)\rangle\neq 0$.
Denoting $\gimel(X)=\langle X1,f\rangle$, this linear 
functional in $\mathcal{L}(H^2(\mathbb{D}))$ satisfies:
\begin{equation}
\label{una}
\gimel(p(M_\varphi))=\langle 
p(M_{h(z^k)})1,f(z^k)\rangle=0
\end{equation}
and
\begin{equation}
\label{dos}
\gimel(M_{g(z^k)})= \langle 
M_{g(z^k)}1,f(z^k)\rangle\neq 0
\end{equation}
for any polynomial $p$. We claim that  $M_{g(z^k)}\in 
\{M_\varphi\}''$. In such a case, equations (\ref{una}) 
and (\ref{dos}) proves that the operator $M_\varphi$ 
do not have the double commutant property.

Indeed, by Baker-Deddens-Ullman's result we know that
$\{M_\varphi\}'=\{M_{z^k}\}'$, therefore 
$\{M_\varphi\}''=\{M_{z^k}\}''$. By Corollary \ref{corte} 
we obtain that $M_{g(z^k)}\in \overline{\textrm{Alg}
(M_{z^k})}^{WOT}$, and since $M_{z^k}$ has the double 
commutant property we get that $M_{g(z^k)}\in 
\{M_{z^k}\}''=\{M_\varphi\}''$ as we desired. Therefore 
we have shown that $M_\varphi$ do not have the double 
commutant property.
\end{proof}

By the argument principle, we can deduce from Proposition \ref{geometric} the following special case included in Corollary \ref{geometric2}. 
Nevertheless, instructively, we include an alternative proof that uses the pigeonhole principle.

\begin{corollary}
\label{geometric2}
Assume that $\varphi$ is an entire function with
$k(\varphi)=k$. If the image of a point under $\varphi$ has $p>k=k(\varphi)$ pre-images, then $M_\varphi$ does
not have the double commutant property.
\end{corollary}
\begin{proof}

Assume that the complex number
$c\in\varphi(\mathbb{D})\subset \mathbb{C}$ has $p$ preimages ($p>k$), that  is,  there exist 
$a_1,\cdots,a_p\in \mathbb{D}$ such that
$\varphi(a_1)=\varphi(a_2)=\cdots =\varphi(a_p)=c$. 
Since $M_\varphi-c I=M_{\varphi-c}$ and the commutant 
is invariant under translations by the identity, we can
assume without loss that 
$\varphi(a_1)=\cdots=\varphi(a_p)=0$.

Since $\varphi$ is an entire function, by Baker-Deddens-Ullman's result, we can deduce that
$\{M_\varphi\}'=\{M_{z^k}\}'$ for $k=k(\varphi)$. 
Therefore, $\{M_\varphi\}''=\{M_{z^k}\}''$. In particular,
$M_{z^k}\in \{M_\varphi\}''$. 

We claim that $M_{z^k}\notin \overline{\textrm{Alg}
(M_\varphi)}^{WOT}$. Assume by contradiction that there
exists a net $p_d(\cdot)$, $d\in D$ such that for any 
$h,g\in H^2$:

$$\lim_d \langle p_d(M_\varphi)^\star 
h,g\rangle=\langle M_{z^n}^\star h,g\rangle.$$

Set $k_a(z)=1/(1-\overline{a}z)$ the reproducing kernel
at $a\in \mathbb{D}$.
Since $k_{a_j}\in \textrm{Ker}(M_\varphi^\star)$ for  
$j=1,\cdots, p$, we get:
$$
\lim_d\langle p_d(M_\varphi)^\star 
(k_{a_j}),1\rangle=\lim_d p_d(0) =\langle 
M_{z^{k(\varphi)}}^\star 
(k_{a_j}),1\rangle=\overline{a_j}^{k(\varphi)}.
$$
But, since $p>k(\varphi)\geq 1$, by using the 
$k(\varphi)$-valence of $z^{k(\varphi)}$ and the 
pigeonhole principle, there exist at least two values 
$1\leq i_0<j_0\leq p$ such that 
$a_{i_0}^{k(\varphi)}\neq 
a_{j_0}^{k(\varphi)}$, a contradiction.
\end{proof}

The equality $\varphi(z)=h(z^k)$ in the Baker-Deddens-
Ullman result suggests that the double commutant
property for an analytic Toeplitz operator induced for an entire function will be closely related to the
approximation of the polynomials on $h(z ^k)$. But in this case this approximation cannot be in $H^\infty(\mathbb{D})$ because such functions are not univalent. 
This fact suggests to consider small spaces. In fact, given $g$ an inner function, let us denote
$$
H^{\infty} [g]=\{f\circ g\,:\,f\in H^\infty(\mathbb{D})\}. 
$$
We say that the polynomials of $\varphi(g)$ are weak star dense in $H^\infty [g]$ if the weak star closure of polynomials on $\varphi(g)$ contains the space $H^\infty[g]$, and in such a case we will say that $\varphi(g)$ is a weak star generator on $H^\infty[g]$.



\begin{lemma}
\label{sub}
Assume that $g$ is an inner function and suppose that $g(\mathbb{D})=\mathbb{D}$. The function
$\varphi$ is a weak star generator of
$H^\infty(\mathbb{D})$ if and only if $\varphi(g)$ is a
weak star generator of $H^\infty[g]$.
\end{lemma}
\begin{proof}
We denote $$\mathcal{P}=
\{p(\varphi)\,:\,p\,\textrm{polynomial}\}$$
and $$\mathcal{P}[g]=\{
p(\varphi(g))\,:\,p\,\textrm{polynomial}\}.$$

For $\alpha=1$, $h\in \mathcal{P}_*^{(1)}$  if and only if 
there 
exists a sequence of polynomials $p_n$ such that
\begin{equation}
\label{cond}
\|p_n(\varphi)\|_\infty \leq M\quad \textrm{and}\quad 
p_n(\varphi (a))\to h(a)\,\, \textrm{for each}\,a\in 
\mathbb{D},
\end{equation}
which implies that 
$\|p_n(\varphi(g))\|_\infty \leq M$ and 
$p_n(\varphi(g(a))\to h(g(a))$ for each $a\in \mathbb{D}$, therefore $h\in \mathcal{P}_{*}^{(1)}[g]$.

Conversely, if 
$\|p_n(\varphi(g))\|_\infty \leq M$ and
$p_n(\varphi(g(a))\to h(g(a))$ for each $a\in 
\mathbb{D}$, then   
$\|p_n(\varphi)\|_\infty \leq M$ and 
$p_n(\varphi(b))\to h(b)$ for each $b\in g(\mathbb{D})=\mathbb{D}$. That is, 
$h\in \mathcal{P}_{*}^{(1)}$.

The result follows because inductively we get that for
any ordinal $\alpha$ we have $h\in 
\mathcal{P}_{*}^{(\alpha)}$ if and only if $h(g)\in 
\mathcal{P}_{*}^{(\alpha)}[g]$.
\end{proof}

Let $\varphi\in H^\infty(\mathbb{D})$, we say that $\varphi$ is in the Thomson-Cowen class ($\mathcal{TC}(\mathbb{D})$) if there exists a point $\lambda\in \mathbb{D}$ such that the inner part of $\varphi-\varphi(\lambda)$ is a finite Blaschke product. It is well known that the class $\mathcal{TC}(\mathbb{D})$ contains all non-constant functions in $H^\infty(\overline{\mathbb{D}})$.
In the following, we state the remarkable theorem due to Thomson and Cowen (\cite{cowen,thomson2,thomson}):

\begin{theorem}[Thomson(1976)-Cowen(1978)]
Assume $\phi\in \mathcal{TC}(\mathbb{D})$. Then there exist a finite Blaschke product $B$ and a function $\psi\in H^\infty(\mathbb{D})$ such that $\phi=\psi(B)$ and $\{M_\phi\}'=\{M_B\}'$.
\end{theorem}

The following result characterizes when an element in the Thomson-Cowen class has the double commutant property.

\begin{theorem}
\label{entera}
Assume that $\varphi\in \mathcal{TC}(\mathbb{D})$ and set $\varphi=h(B)$ as a Thomson-Cowen factorization of $\varphi$.  The following conditions are equivalents
\begin{enumerate}
\item[a)] $M_\varphi$ has the double commutant property.
\item[b)] The polynomials on $\varphi$ are weak star dense in $H^{\infty}[B]$.
\end{enumerate}
\end{theorem}
\begin{proof}
By way of contradiction let us suppose that $M_\varphi$ has the double commutant property and that  the polynomials on $h(B)$ are not weak star dense in $H^\infty[B]$. As in the proof of Theorem \ref{main1} (see Remark \ref{refesub}) we get that the polynomials on $h(B)$ are not WOT dense in $H^\infty[B]$. Hence,
there exist  functions $f\in H^\infty(\mathbb{D})$ and $\psi(B)\in H^\infty[B]$ such that for any polynomial $p$
\begin{equation}
\label{one}
\int_{0}^{2\pi} p\circ h(B(e^{it})) f(e^{it}) \,dt=0
\end{equation}
and
\begin{equation}
\label{two}
\int_{0}^{2\pi} \psi(B(e^{it})) f(e^{it}) \,dt\neq 0.
\end{equation}

Let us consider the linear functional $\gimel$ defined on the subspace $\{M_{g}\,:\,g\in H^\infty(\mathbb{D})\}$
as $$\gimel(M_g)=\int_{0}^{2\pi} g(e^{it}) f(e^{it})\,dt$$
and the sublinear function $p(T)=\|T\|$. Let us denote $\widetilde{\gimel}$ the Hahn-Banach's sublinear extension of $\gimel$. By (\ref{one}) and (\ref{two})  we have that $\widetilde{\gimel}(M_{p\circ h(B)})=\widetilde{\gimel}(M_{p(\varphi)})=0$ for any polynomial $p$ and $\widetilde{\gimel}(M_{\psi(B)})\neq 0$. On the other hand, by Corollary \ref{corte} we get that  $M_{\psi(B)}\in \{M_{B}\}''=\{M_\varphi\}''$, that is, we arrive to a contradiction, because by hypothesis $M_\varphi$ has the double commutant property.

Assume that the polynomials on $\varphi$ are weak star dense in $H^\infty[B]$. Set $X\in \{M_\varphi\}''$. 
Since $\varphi=h(B)$, we have that  $\{M_{B}\}''=\{M_\varphi\}''$. Therefore,
by Corollary 
\ref{doblen2} we get that 
$X=M_{\psi(B)}$ for some $\psi\in 
H^\infty(\mathbb{D})$. 

Indeed, let us show that for 
any ordinal $\alpha$
$$
\{M_g\,:\,g\in \mathcal{P}_*^{(\alpha)}
[B]\}\subset \overline{
\{M_{p(\varphi)}\, 
p\,\textrm{polynomial}\}}^{WOT}:=\mathcal{M}.
$$
For $\alpha=1$, if $g\in 
\mathcal{P}_*^{(1)}[B]$ there exists 
a sequence of polynomials $p_n$ such 
that
$\|p_n(h(B))\|_\infty \leq C$ and 
$p_n(h(B(a)))\to g(a)$ for any $a\in 
\mathbb{D}$. By Remark \ref{trans2} 
we get that $M_{p_n(h(B))}$ 
converges in the strong operator 
topology to $M_g$, therefore $M_g\in 
\mathcal{M}$.
Now, let us assume that for an 
ordinal $\alpha$ we have
$\{M_g\,:\,g\in 
\mathcal{P}^{(\alpha)}[B]\}\subset 
\mathcal{M}
$. If $g\in \mathcal{P}_*^{(\alpha+1)}$ 
there exists a sequence of functions 
$g_l\in \mathcal{P}_*^{(\alpha)}
[B]$, such that 
$\|g_l\|_\infty<C$ and 
$g_l(a)\to g(a)$ for any $a\in 
\mathbb{D}$. By Remark \ref{trans} 
we have that $M_{g_l}\to M_g$ in the strong operator topology and $M_{g_l}\in \mathcal{M}$ for all $l$, since $\mathcal{M}$ is closed in the weak operator topology we get that $M_g\in \mathcal{M}$. 

Since there exists an ordinal
$\beta$ such that $\mathcal{P}_*^{(\beta)}[B]\supset H^\infty[B]$. We get that 

$$
\{M_{g(B)}\,:\,g\in H^\infty(\mathbb{D})\}\subset 
\mathcal{M}=\overline{
\{M_{p(\varphi)}\, :\,
p\,\textrm{polynomial}\}}^{WOT}.
$$
That is, $X=M_{\psi(B)}\in \overline{
\{M_{p(\varphi)}\,:\,
p\,\textrm{polynomial}\}}^{WOT}$ as we desired to prove.
\end{proof}

\begin{remark}
If we consider $\varphi(z)=z^6$ and $f(z)=z^2$, clearly $\varphi(z)=f(z^3)$. Of course $\varphi(z)=z^6$ has the double commutant property, but the polynomials on $\varphi$ are not weak star dense in $H^{\infty} [z^3]$ and the polynomials on $z^2$ are not weak star dense in $H^\infty(\mathbb{D})$. At first glance, one may think that something is wrong. The problem here is that the representation $\varphi(z)=z^6=f(z^3)$ is not the Thomson-Cowen representation of $\varphi$, because the commutant of $M_\varphi$ is larger than the commutant of $M_{z^3}$. In this case, the representation of $z^6$ is just $z^6=f(z^6)$ when $f(z)=z$. That is, the equality of the commutants $\{M_\varphi\}'=\{M_B\}'$ in the Thomson-Cowen's Theorem is fundamental.
\end{remark}

\begin{corollary}
Assume that $\varphi\in \mathcal{TC}(\mathbb{D})$, 
and $\varphi(z)=h(B)$ with $B$ a finite Blaschke
product.  The following conditions are equivalents:
\begin{enumerate}
\item[a)] $M_\varphi$ has the double commutant property.
\item[b)] $M_h$ has the minimal commutant property.
\end{enumerate}
\end{corollary}

Now, let us obtain the geometric result for functions in the Thomson-Cowen class. It is well known that factorization $\varphi=h(B)$, is unique in the sense of Moebius modulo maps. We denote by $b(\varphi)$ the Thomson-Cowen order of $\varphi$, that is, the maximal order of the finite Blaschke product $B'$ for which $\varphi=h'(B')$ for some $h'\in H^\infty(\mathbb{D})$.  For maps $\varphi\in H^\infty(\overline{\mathbb{D}})$ it is defined the minimal winding number as
$$
k(\varphi)=\inf
\{n(\varphi(\gamma),\varphi(a))\,:\,\varphi(a)\notin \varphi(\partial \mathbb{D})\}.
$$
The map $\varphi\in H^\infty(\overline{\mathbb{D}})$ is said to have the Minimal Winding number property if $k(\varphi)=b(\varphi)$. Clearly, by Baker-Deddens and Ullmann's result, if $\varphi$ is an entire function, then $\varphi$ has the Minimal Winding number property.

The following result is analogous to Proposition \ref{geometric} for the functions in $H^\infty(\overline{\mathbb{D}})$.

\begin{theorem}
\label{geo2}
Assume that $\varphi\in H^\infty(\overline{\mathbb{D}})$. If there are two points $a,b\in\mathbb{D}$ such that $\varphi(a),\varphi(b)\notin \varphi(\partial \mathbb{D})$ and $n(\varphi(\gamma),\varphi(a))\neq n(\varphi(\gamma),\varphi(b))$, then $M_\varphi$ has not the double commutant property.
\end{theorem}
\begin{proof}
In fact, we can suppose without loss that $\varphi$ is not constant. In such a case $\varphi\in \mathcal{TC}(\mathbb{D})$. Set $\varphi(z)=h(B)$ as the Thomson-Cowen factorization. If $M_\varphi$ has the double commutant property, by Theorem \ref{entera} the polynomials on $h$ are weak star dense in $H^\infty (\mathbb{D})$, therefore, by \cite[Proposition 3]{sarasondos} we find that $h$ is univalent. Therefore, for any $a\in \mathbb{D}$ such that $\varphi(a)\notin \varphi(\partial \mathbb{D})$ we have $n(\varphi(\gamma),\varphi(a))=b(B)$, a contradiction.
\end{proof}
\begin{corollary}
Assume that $\varphi\in  
H^\infty(\overline{\mathbb{D}})$. If $M_\varphi$ 
has the double commutant property then $\varphi$ 
has the Minimal Winding number property. 
\end{corollary}

\begin{corollary}
Assume that $\varphi\in H^\infty(\overline{\mathbb{D}})$ and that $\varphi$ acting on $\partial\mathbb{D}$ is a Jordan curve. Then $M_\varphi$ has the double commutant property if and only if $n(\varphi(\gamma),\varphi(a))$ is constant for each $a\in\mathbb{D}$ such that $\varphi(a)\notin \varphi(\partial \mathbb{D})$.
\end{corollary}
\begin{proof}
The necessity part follows from Theorem \ref{geo2}. Now, we assume that $\varphi(z)=f(B)$. By Theorem \ref{entera} we see that $f$ is univalent. Since the boundary of $f(\mathbb{D})$ is a Jordan curve, we find that the polynomials on $f$ are weak star dense in $H^\infty(\mathbb{D})$, therefore, by Walsh Theorem $M_\varphi$ has the minimal commutant property. By Theorem \ref{entera} again, we find that $M_\varphi$ has the double commutant property as we wanted. 
\end{proof}

The following example shows that the geometric condition in Theorem \ref{geo2}, in general, is not sufficient to guarantee the double commutant property
of an analytic Toeplitz operator. 
\begin{example}
Let $f$ be a univalent function that maps $\mathbb{D}$ to the slit disk $\mathbb{D}\setminus [0,1)$. If we
consider $\varphi(z)=f(z^p)$ we have that for any $c\in 
\varphi(\mathbb{D})$, $M_\varphi$ do not have the double commutant property.
\end{example}

\section{Concluding remarks and open questions}

There are examples of bounded analytic functions that are not in the Thomson-Cowen class, such as any infinite Blaschke product. This example was provided by the referee $\exp(-s)$ where $s=(1-z)/(1+z)$. What is not at all trivial is whether the functions of the disk algebra are included in the Thomson-Cowen class (see in \cite[p. 55]{chinos2} this question formulated). A question that awaits to be answered is characterizing the elements of the disk algebra that induce multiplication operators with the double commutant property.

The necessary condition on the winding number is somewhat unexpected. We would like to know whether it is sufficient or not for the case of entire functions. We suspect that it is not a sufficient condition. A good counterexample is to find an entire function that takes the disk to a moon-shaped region such that the kissing point has very strong contact (see \cite[p. 22]{gaier}). Professor Alicia Cantón Pire suggested a good starting point for this question by considering the exponential function $\exp(\pi i z)$. This function is injective on the unit disk and takes the unit disk to a moon-shaped region with a very strong contact in the kissing point.

\vskip .5cm
\textbf{Acknowledgments.}
The authors thank the referees who, although the first version was full of typos, believed in this work. Also,thanks to Professor Alicia Cantón Pire for her time, listening to our questions.

\vskip .5cm

\textbf{Funding information.} M.J.G. was supported in part by the Spanish Ministerio de Ciencia e Innovación (grant no. PID2021-123151NB-I00).
This publication is part of the project PID2022-139449NB-I00, funded by: \\ MCIN/AEI/10.13039/501100011033/FEDER, UE.
The authors were supported by
Grant “Operator Theory: an interdisciplinary approach,” 
reference ProyExcel$\_$00780, a project financed in the 2021 call for Grants for
Excellence Projects, under a competitive biddingregime, aimed at entities qualified
as Agents of the Andalusian Knowledge System, within the scope of the Andalusian Research, Development, and Innovation Plan (PAIDI 2020). Counseling of the
University, Research and Innovation of the Junta de Andalucía.

\section*{Availability of data and material}
Not applicable.

\section*{Ethics Declarations}
Not applicable.
\section*{Author's contribution}
The author contributed significantly to analysis and manuscript preparation and she helped perform the analysis with constructive discussions. 
\section*{Competing interests}
The authors declare that they have no competing interests.

\end{document}